\documentclass[draft,reqno,9pt,article]{amsart}
\usepackage{amsfonts}
\usepackage{amssymb}
\usepackage{amsmath}
\usepackage{amsthm}
\usepackage{latexsym}
\usepackage{graphicx}
\usepackage{layout}
 \usepackage{color}

\newtheorem{theorem}{Theorem}[section]

\newtheorem{proposition}[theorem]{Proposition}

 \renewcommand{\(}{\left(}
\renewcommand{\)}{\right)}

\newcommand{\eps}{\epsilon}

\newcommand{\rr}{ \mathbb{R}}

\begin{document}
\title[The Ljapunov-Schmidt reduction for some critical problems]{The Ljapunov-Schmidt reduction for some critical problems}

\author{Angela Pistoia}
\address{Angela Pistoia, Dipartimento di Metodi e Modelli Matematici,
Universit\`a di Roma ``La Sapienza'', via Antonio Scarpa 16, 00161 Roma,
Italy}
\email{pistoia@dmmm.uniroma1.it}
\maketitle

 \section{Introduction}

Let us consider the problem
\begin{equation} \label{eq1}
\left\{
\begin{aligned}
&-\Delta u=|u|^{q-1 }u & \hbox{in}\ \Omega  ,\\
&u=0 & \hbox{on}\ \partial\Omega, \\
\end{aligned}
 \right.
\end{equation}
where $\Omega$ is a smooth bounded   domain in $\rr^n,$ $n\ge3$ and $q>1.$ Let $2^*$ denote  the critical exponent in the Sobolev embeddings, i.e. $2^*=\frac{2n}{n-2}$.

In the subcritical case, i.e.  $q<2^*-1$  compactness of Sobolev's embedding ensures existence of at least one  positive solution  and infinitely many sign changing solutions to \eqref{eq1}.

In the  critical case or in the supercritical case, i.e.  $q\ge2^*-1$ existence of solutions is a delicate issue.
In \cite{poho} Poho\u{z}aev  proved that  the problem \eqref{eq1}
 does not admit a nontrivial
solution if $\Omega$ is star-shaped. On the other hand,  Kazdan and Warner in \cite{kaz} proved that problem \eqref{eq1} has
one positive radial solution and infinitely many sign changing radial solutions if  $\Omega$ is an annulus.
In the critical case, i.e. $q=2^*-1$, Bahri and Coron in  \cite{baco} found a positive solution to  \eqref{eq1}   provided
  the domain $\Omega$ has a \textit{nontrivial topology}.
\medskip

In this survey we are in particular interested in the following  perturbed critical problems.

\begin{itemize}
\item[$\bullet$] {\em The Brezis-Nirenberg  problem }
 $$  (\mathcal{BN})_\eps\qquad\qquad\left\{
\begin{aligned}
&-\Delta u=|u|^{2^*-2}u+\epsilon u & \hbox{in}\ \Omega\subset\rr^n,\ n\ge4,\\
&u=0 & \hbox{on}\ \partial\Omega\\
\end{aligned}
 \right.
 $$
\item[$\bullet$] {\em The ''almost-critical" problem }
  $$ (\mathcal{AC})_\eps\qquad\qquad\left\{
\begin{aligned}
&-\Delta u=|u|^{2^*-2-\epsilon}u & \hbox{in}\ \Omega,\\
&u=0 & \hbox{on}\ \partial\Omega\\
\end{aligned}
 \right.
 $$
 \item[$\bullet$] {\em The Coron's problem }
$$ (\mathcal{C})_\eps\qquad\qquad \left\{
\begin{aligned}
&-\Delta u=|u|^{2^*-2 }u & \hbox{in}\ \Omega _\eps:=\Omega\setminus B(z_0,\eps),\\
&u=0 & \hbox{on}\ \partial\Omega_\eps\\
\end{aligned}
 \right.
 $$
\end{itemize}
   In the first two problems  $\eps\in\rr$ is a small parameter either positive or negative. In the last problem $z_0\in \Omega $ and $\eps$ is a small positive parameter.

   The common feature of those problems is that when $\eps$ is small enough they can have    solutions $u_\eps$ whose shape  resembles the sum of a finite number of bubbles as $\eps$ goes to zero, i.e.
   $$u_\eps(x)\sim \sum\limits_{i=1}^k\lambda_iU_{\delta_i^\eps ,z_i ^\eps}(x)$$
   where the concentration points $z_i^\eps$ converge to  a point $z_i^0$ in  $\Omega$ and the concentration parameters $\delta_i^\eps$ go to zero as $\eps$ go to zero.
If  $\lambda_i=+1$ ($\lambda_i=-1$) we say that $u_\eps$ has a positive (negative) blow-up point at  $z_i^0$ as $\eps$ goes to zero.  

  A {\em bubble} is a function
\begin{equation}\label{udz}
U_{\delta ,z }(x):=\alpha_n {\frac{\delta ^{\frac{%
n-2}{2}}}{(\delta ^{2}+|x-z |^{2})^{\frac{n-2}{2}}},}\qquad \delta >0,%
\text{\quad }x,z \in \mathbb{R}^{n}.
\end{equation}%
Here $\alpha _{n}:=\left[ n(n-2)\right] ^{\frac{n-2}{4}}.$ They are  positive solutions to the limit problem (see Aubin \cite{au}, Caffarelli-Gidas-Spruck \cite{cgs}, Talenti \cite{tal})
\begin{equation}\label{lp}
-\Delta u=u^{n+2\over n-2} \quad \hbox{in}\ \mathbb{R}^{n} .
\end{equation}%

  The location of the points $z_i$'s where blowing-up occurs is strictly related to the geometry of the domain,
  namely Green's and Robin's functions.
Let   $G$ the Green's function of the negative laplacian on
$\Omega$ with Dirichlet boundary conditions
and let $H $ its regular part, i.e.
$$
H (x,y)=\frac{c_n}{|x-y|^{n-2}} -G (x,y),\quad \forall (x,y)\in
\Omega^2,
$$
where $c_n$ is a positive constant.
The function $\tau(x):= H (x,x),$ $x\in\Omega$ is called
  {\em Robin's function.}  It is known that  $\tau$ is a $C^2-$function and also that $\tau(x)$ goes to $+\infty$ as $x$ approaches the boundary of $\Omega.$ Therefore, the Robin's function has always a minimum point in $\Omega.$

\medskip
Let us state the main results concerning solutions to  $ (\mathcal{BN})_\eps,$  $ (\mathcal{AC})_\eps$ and  $ (\mathcal{C})_\eps$ which blow-up at one or more points of the domain as the parameter $\eps$ goes to zero.
\bigskip

  {\em The Brezis-Nirenberg  problem   and  the ''almost critical" problem when $\eps$ is positive}

  Brezis and Nirenberg in \cite{bn} proved that if $n\ge4 $  for
  small enough problem $ (\mathcal{BN})_\eps$ has a   positive solution  provided $\eps$ is small enough.
On the other hand, it is clear that the slightly sub-critical problem  $ (\mathcal{AC})_\eps$ has always a positive solution.
Han in \cite{h}  proved that these solutions blow-up at a critical point of the Robin's function as $\eps$ goes to zero.
Conversely, Rey in \cite{rey1,rey2} proved that any $C^1-$stable critical point $z_0$ of the Robin's function generates a family of solutions which blows-up at $z_0$  as $\eps$ goes to zero.
Musso-Pistoia in \cite{mupi4} and Bahri-Li-Rey in \cite{blr} studied existence of solutions which blow-up at $\kappa$ different points of $\Omega.$
Grossi-Takahashi \cite{grotak} proved the nonexistence of positive
solutions  blowing up at $\kappa\ge2$   points for these problems in convex domains.

As far as it concerns the existence of sign changing solutions, the slightly sub-critical problem  $ (\mathcal{AC})_\eps$ has infinitely many sign changing solutions. Existence of sign changing solution for problem $ (\mathcal{BN})_\eps$ is a more difficult problem.
The first result about problem $ (\mathcal{BN})_\eps$ is due to Cerami-Solimini-Struwe, who showed in \cite{css} the existence of a pair of
least energy  sign changing solutions  if $n\ge6$ and $\eps$ is small enough.
The existence of infinitely many  solutions to $ (\mathcal{BN})_\eps$    for any
$\eps>0$ was established by Devillanova-Solimini in \cite{ds1} when $n\ge7.$  Moreover, for low dimensions 
$n =4,5,6,$ in
\cite{ds2} they proved the existence of at least $n+1$ pairs of solutions   provided
$\eps$ is small enough.
Ben Ayed-El Mehdi-Pacella in \cite{bep1,bep2}  studied   the blow up of the low energy sign-changing solutions of problems $ (\mathcal{AC})_\eps$ and $ (\mathcal{BN})_\eps$ as $\eps$ goes to zero
and they   classified these solutions according
to the concentration speeds of the positive and negative part.
In \cite{cc} Castro-Clapp  proved the existence of one pair of solutions in a symmetric domain   which change sign exactly once, provided  
$n\ge4$ and $\eps$ is small enough. Moreover they describe the profile of the solutions, by showing that the
solutions blow-up positively and negatively at two different points in $\Omega$ as $\eps$ goes to 0.
Micheletti-Pistoia in \cite{mipi1} and Bartsch-Micheletti-Pistoia in \cite{bamipi} generalized such a result showing the existence of at least $n$ pairs of sign changing solutions with one negative and one positive blow-up points.
Pistoia-Weth in \cite{piwe} and Musso-Pistoia in \cite{mupi5} proved that a large number of sign changing solutions exists for problem $ (\mathcal{AC})_\eps$: the solutions are a superposition with alternating sign of bubbles whose centers collapse to the minimum point of the Robin's function as $\eps$ goes to zero.
This result is unknown for problem  $ (\mathcal{BN})_\eps$ even if we think to be true.
\medskip

  {\em The Brezis-Nirenberg  problem   and  the ''almost critical" problem when $\eps$ is negative}

In \cite{poho} Poho\u{z}aev  proved that problems $ (\mathcal{B})_\eps$ and $ (\mathcal{AC})_\eps$ do not have any solutions if $\eps$ is negative and $\Omega$ is starshaped.

As far as it concerns the existence of blowing-up solutions, when $\eps$ is negative  and small enough completely different phenomena take place even if the domain $\Omega$ is not starshaped. 
Indeed, Ben Ayed-El Mehdi-Grossi-Rey in   \cite{begr} proved that  problem $ (\mathcal{AC})_\eps$ do not have any positive solutions which blows-up at one point when $\eps$ goes to zero. We believe that their argument could also be extended  to
the problem   $ (\mathcal{BN})_\eps.$
  Del Pino-Felmer-Musso in \cite{delfemu3} and Musso-Pistoia in \cite{mupi1} found,  for $\eps$   small enough, a positive solutions with two positive blow-up points provided
the domain $\Omega$ has a hole.
Del Pino-Felmer-Musso in \cite{delfemu1,delfemu2} and Pistoia-Rey in \cite{pire} found solutions with three or more positive blow-up points, under suitable assumptions on the domain $\Omega .$  Towers of positive bubbles were  constructed by del Pino-Dolbeault-Musso in \cite{ddm,ddm1} and by Ge-Jing-Pacard in \cite{gjp}, under suitable assumptions on non degeneracy of Robin's    and   Green's functions.
 As far as it concerns the study of sign changing solutions, Ben Ayed-Bouh in \cite{bb} proved that problem      $ (\mathcal{AC})_\eps,$
does not have any sign changing solutions with one  positive and one  or two negative blow-up points. We believe that their argument could also be extended also to
the problem   $ (\mathcal{BN})_\eps.$
 There are no results about existence of sign changing solutions for these problems.

  \medskip

    {\em The Coron's problem }
 Coron in \cite{c}
found via variational methods a solution to problem  $ (\mathcal{C})_\eps $ provided $\eps$ is small enough.
 If the domain
 has several holes,  Rey in \cite{rey3} and   Li-Yan-Yang in \cite{lyy} constructed solutions blowing-up at the centers of the holes
 as the size of the holes goes to zero.   On the other hand, Clapp-Weth
in \cite{cw}  found a second solution to $ (\mathcal{C})_\eps $, but they were unable to say if it 
 was positive or changed sign.
  Clapp-Musso-Pistoia in \cite{cmp} found positive and sign changing solutions to $ (\mathcal{C})_\eps $
blowing-up at the center of the hole and at  one or more points inside the domain as $\eps$ goes to zero.
If the domain has two small holes, Musso-Pistoia in \cite{mupi3} constructed a sign changing solution
  with one positive blow-up point and one negative blow-up point at the centers of the two holes.
Musso-Pistoia in \cite{mupi2} and Ge-Musso-Pistoia in \cite{gmp} found a large  number of sign changing solutions to  $ (\mathcal{C})_\eps $:
the solutions are a superposition of bubbles with alternating sign whose centers collapse to the center of the hole as $\eps$ goes to zero.

\medskip
 The proofs of all the results concerning existence of solutions which blow-up positively or negatively at one or more points as the parameter $\eps$ goes to zero,
rely  on a Lyapunov-Schmidt reduction scheme firstly developed by Bahri-Coron in \cite{baco}.
This allows to reduce
the problem of finding  blowing-up solutions  to the problem of finding
critical points of a functional which depends only on the blow-up points  and the concentration rates.
 The leading part of the reduced functional is explicitly given in terms of the geometry of the domain, namely  
Green's and Robin's functions. The reduced functional also takes into account
the different interactions among the bubbles  which depends on their respective sign.  Finally, we use a variational approach and we
obtain the existence of  critical  points of the reduced functional by applying a minimization argument or a min-max argument.
In the following we describe the main steps to get   some of the previous results.
We will refer to
  \cite{mupi4} and  \cite{blr,delfemu3,mipi1} for the proofs related to the construction of positive and sign-changing  multi-bubbles   to problems $(\mathcal{BN})_\eps$ and $ (\mathcal{AC})_\eps$, respectively. We will refer to \cite{mupi5} and  to  \cite{mupi2,gmp}
for the proofs related to the construction of towers of bubbles    to problems $ (\mathcal{AC})_\eps$ and $ (\mathcal{C})_\eps$, respectively.

\section{Setting of the problem}

We want to rewrite problems  $(\mathcal{BN})_\eps$, $(\mathcal{AC})_\eps$ and $(\mathcal{C})_\eps$ in a different, but equivalent, form.

\medskip
Let us take
\begin{equation*}
(u,v):=\int_{\Omega } \nabla u\cdot \nabla v\text{ }dx,\text{\qquad }%
\Vert u\Vert :=\left( \int_{\Omega } \left\vert \nabla u\right\vert
^{2}dx\right) ^{1/2},
\end{equation*}%
as the inner product in $\mathrm{H}_{0}^{1}(\Omega )$ and its corresponding
norm.   Similarly, for each $r\in
\lbrack 1,\infty )$,%
\begin{equation*}
\Vert u\Vert _{r}:=\left( \int_{\Omega } \left\vert u\right\vert
^{r}dx\right) ^{1/r}
\end{equation*}%
is a norm in $\mathrm{L}^{r}(\Omega ).$

Let $i^{\ast }:%
\mathrm{L}^{\frac{2n}{n+2}}(\Omega )\rightarrow \mathrm{H}_{0}^{1}(\Omega )$
be the adjoint operator to the embedding $i:\mathrm{H}_{0}^{1}(\Omega
)\hookrightarrow \mathrm{L}^{\frac{2n}{n-2}}(\Omega ),$ i.e. $i^{\ast }(u)=v$
if and only if%
\begin{equation*}
(v,\varphi )=\int_{\Omega } u(x)\varphi (x)dx\quad \text{for all }\varphi
\in C_{c}^{\infty }(\Omega )
\end{equation*}%
if and only if%
\begin{equation*}
-\Delta v=u\quad \text{in}\ \Omega ,\qquad v=0\quad
\text{on}\ \partial \Omega .
\end{equation*}
It is clear that $i^*$ is a continuous map, namely there exists a positive constant $c$ such that
\begin{equation*}
\left\Vert i^{\ast }(u)\right\Vert \leq c\left\Vert u\right\Vert _{\frac{2n}{%
n+2}}\quad \forall \ u\in \mathrm{L}^{\frac{2n}{n+2}}(\Omega ).
\end{equation*}%

To study the slightly supercritical case, namely problem $(\mathcal{AC})_\eps$ with $\eps<0,$   we need to find solutions in the space
$ \mathrm{H}^1_0(\Omega)\cap \mathrm{L}^{s_\varepsilon}(\Omega)$ with
 $s_\eps:={2n\over n-2}-\eps{n\over2}.$
 Indeed, by a well known Hardy-Littlewood-Sobolev inequality (see Hardy-Littlewood \cite{hl} and Sobolev \cite{s})
we deduce that $i^\ast$ restricted to $ \mathrm{H}^1_0(\Omega)\cap \mathrm{L}^{s_\varepsilon}(\Omega)$ is a continuous map, namely
$$\left \|i ^\ast (u)\right \| _{s}\le c \left \|u\right \|_{s_\eps\over {n+2\over n-2}-\eps}
$$ for some positive constant $c $ which depends only on $n.$
We point out that if $\eps>0$ then $\mathrm{H}^1_0(\Omega)\cap \mathrm{L}^{s_\varepsilon}(\Omega)$ coincides with $\mathrm{H}^1_0(\Omega).$

Using the above definitions and notations, it is clear that our problems can be rewritten in the equivalent form

\begin{equation}\label{rep}
\left\{
\begin{aligned}
&u=i ^\ast \left[f_\eps(u)    \right] \\
& u\in \mathfrak{H}_\eps,
\\
\end{aligned}\right.\end{equation}

where

\begin{itemize}
\item[(i)] {\em  for problem $(\mathcal{BN})_\eps$}
$$ f_\eps(u):=|u|^{p-1 }s+\eps u \quad \hbox{and}\ \mathfrak{H}_\eps:=\mathrm{H}^1_0(\Omega) $$

\item[(ii)] {\em   for problem $(\mathcal{AC})_\eps$}
$$f_\eps(u):=|u|^{p-1-\eps}u  \quad \hbox{and}\  \mathfrak{H}_\eps:=\mathrm{H}^1_0(\Omega)\cap \mathrm{L}^{s_\varepsilon}(\Omega)
$$

\item[(iii)] {\em   for problem $(\mathcal{C})_\eps$}
$$f_\eps(u):=|u|^{p-1 }u  \quad \hbox{and}\  \mathfrak{H}_\eps:=\mathrm{H}^1_0(\Omega_\eps).
$$

\end{itemize}

\medskip

\section{The Ljapunov-Schmidt procedure}
\subsection{The approximating solution}

The first step  is writing a {\em good approximating solution.}
\medskip

Let $PW$ denote the projection of the function $W\in D^{1,2}(\mathbb{R}^{n})$
onto $\mathrm{H}_{0}^{1}(D )$, i.e.
\begin{equation*}
\Delta PW=\Delta W\ \text{\ in}\ D ,\qquad PW=0\ \text{\ on}\ \partial
D ,
\end{equation*}
where $D$ is a smooth bounded domain in $\rr^n.$

Let $\kappa\ge1$ be a fixed integer.
 We look for   solutions $u_\eps$ to problems $(\mathcal{BN})_\eps$, $(\mathcal{AC})_\eps$ and $(\mathcal{C})_\eps$  as
 \begin{equation}
u_{\epsilon }(x)=V_{\mathbf{z},\mathbf{%
d} }(x)  +\phi(x),\quad  V_{\mathbf{z},\mathbf{%
d} }(x):=\sum\limits_{i=1}^\kappa \lambda_iPU_{\delta _{i},z _{i}} (x) ,
\label{ans}
\end{equation}%
where the higher order term $\phi$ belongs to a suitable space described in the next subsection. Here $\lambda_i\in\{-1,+1\}$, the concentration points $z_i$'s lye in $\Omega$ and the  concentration parameters $\delta_i$'s
are choose as follows.

\begin{itemize}
\item[$ \bullet$] {\em  Multi-bubbles   for problem $(\mathcal{BN})_\eps$}
$$  \mathfrak{M}-(\mathcal{BN})_\eps\qquad  \left\{
 \begin{aligned}
&\lambda_i\in\{-1,+1\}\\
&z _{1},\dots,z_\kappa\in\Omega\ \hbox{and}\ z_i\not=z_j\\
&\delta _{i}=|\epsilon| ^{\frac{1}{n-4}}d_{i}\quad \text{with}\quad
d_{i}>0 \\
\end{aligned}
\right.
$$
The configuration space is
$$\Lambda:=\left\{(\mathbf{z},\mathbf{d})\ :\  \mathbf{z}=(z_1,\dots,z_\kappa)  \in\Omega^\kappa,\ z_i\not=z_j,\ \mathbf{ d}=(d_1,\dots,d_\kappa)\in (0,+\infty)^\kappa \right\}.$$

\item[$\bullet$] {\em  Multi-bubbles  for problem $(\mathcal{AC})_\eps$}
$$\mathfrak{M}-(\mathcal{AC})_\eps\qquad
 \left\{
 \begin{aligned}
&\lambda_i\in\{-1,+1\}\\
&z _{1},\dots,z_\kappa\in\Omega\quad \hbox{and}\quad z_i\not=z_j\\
&\delta _{i}=|\epsilon| ^{\frac{1}{n-2}}d_{i}\quad \text{with}\quad
d_{i}>0. \\
\end{aligned}
\right.
$$
The configuration space is
$$\Lambda:=\left\{(\mathbf{z},\mathbf{d})\ :\  \mathbf{z}=(z_1,\dots,z_\kappa)  \in\Omega^\kappa,\ z_i\not=z_j,\ \mathbf{ d}=(d_1,\dots,d_\kappa)\in (0,+\infty)^\kappa \right\}.$$

\item[$\bullet$] {\em Tower of bubbles with alternating sign  for problem $(\mathcal{AC})_\eps$ when $\eps>0$ }
$$ \mathfrak{T}-(\mathcal{AC})_\eps \qquad
 \left\{
 \begin{aligned}
&\lambda_i=(-1)^i\\
& z_i =z+\delta_i\sigma_i\in\Omega\quad \text{with}\quad  \sigma_1,\dots,\sigma_{\kappa-1} \in\rr^n \ \hbox{and}\ \sigma_\kappa=0 \\
& \delta_i =\eps^{{2(i-1) +1 \over N-2}}  d_i \quad \hbox{with}\  d_i>0. \\
\end{aligned}
\right.
$$
The configuration space is
$$\Lambda:=\left\{(\mathbf{z},\mathbf{d})\ :\  \mathbf{z}=(\sigma_1,\dots,\sigma_{\kappa-1},z)  \in\rr^{(\kappa-1)n}\times\Omega   ,\ \mathbf{ d}=(d_1,\dots,d_\kappa)\in (0,+\infty)^\kappa \right\}.$$

\item[$\bullet$] {\em Tower of bubbles with alternating sign   for problem $(\mathcal{C})_\eps$}
$$\mathfrak{T}-(\mathcal{C})_\eps\qquad
 \left\{
 \begin{aligned}
&\lambda_i=(-1)^i\\
& z_i =z_0 +\delta_i\sigma_i\in\Omega\quad \text{with}\quad  \sigma_1,\dots,\sigma_{k } \in\rr^N  \\
& \delta_i =\eps^{2i-1 \over 2k}    d_i \quad \hbox{with}\  d_i>0. \\
\end{aligned}
\right.
$$
The configuration space is
$$\Lambda:=\left\{(\mathbf{z},\mathbf{d})\ :\  \mathbf{z}=(\sigma_1,\dots,\sigma_{\kappa } )  \in\rr^{ \kappa n}   ,\ \mathbf{ d}=(d_1,\dots,d_\kappa)\in (0,+\infty)^\kappa \right\}.$$

\end{itemize}

  To say that  $V_{\mathbf{z},\mathbf{%
d} }$ is a {\em good approximating solution} we need to estimate the error
\begin{equation*}
\mathfrak{R}_{\mathbf{z},\mathbf{d} }  :=V_{\mathbf{z},\mathbf{%
d} } -i^{\ast }\left[ f _\eps (V_{\mathbf{z},%
\mathbf{d} })  \right] \in \mathfrak{H}_\eps.
\end{equation*}%

\begin{proposition}
\label{err} For any compact subset $\mathbf{C}$ of $\Lambda $ there exist $%
\epsilon _{0}>0$ and $c>0$ such that for each $\epsilon \in (0,\epsilon
_{0}) $ and $(\mathbf{z},\mathbf{d} )\in \mathbf{C}$  we have
$$
\left\Vert\mathfrak{R} _{\mathbf{z},\mathbf{d} } \right\Vert
\leq c    |\eps|^{\eta} .
$$
for some  $\eta>0$ which depends only on $n $ and $\kappa.$
\end{proposition}

\subsection{The equation becomes a system}

The second step  is writing  the equation as a system.
\medskip

We need to fix the space where the rest term $\phi$ in \eqref{ans} belongs to.
 It is important to introduce the functions
\begin{equation*}
\psi _{\delta ,z }^{0}(x):={\frac{\partial U_{\delta ,z }}{\partial
\delta }}=\alpha _{n}{\frac{n-2}{2}}\delta ^{\frac{n-4}{2}}{\frac{|x-z
|^{2}-\delta ^{2}}{(\delta ^{2}+|x-z |^{2})^{n/2}}}
\end{equation*}%
and, for each $j=1,\dots ,n,$
\begin{equation*}
\psi _{\delta ,z}^{j}(x):={\frac{\partial U_{\delta ,z }}{\partial z
_{j}}}=\alpha _{n}(n-2)\delta ^{\frac{n-2}{2}}{\frac{x_{j}-z _{j}}{(\delta
^{2}+|x-z |^{2})^{n/2}}},
\end{equation*}%
which span the set of solutions to the linearized problem (see Bianchi-Egnell \cite{be})
\begin{equation*}
-\Delta \psi =f'_0\(U_{\delta ,z }\)\psi \ \hbox{in}\ \mathbb{R}^{n}.
\end{equation*}%

Remember that the $\delta_i$'s, the $z_i$'s and the configuration space $\Lambda$ are given in $\mathfrak{M}-(\mathcal{BN})_\eps$, $\mathfrak{M}-(\mathcal{AC})_\eps$, $\mathfrak{T}-(\mathcal{AC})_\eps$
and $\mathfrak{T}-(\mathcal{C})_\eps$.

If $(\mathbf{z},\mathbf{d})\in\Lambda$,
  we introduce the spaces
\begin{equation*}
K_{\mathbf{z},\mathbf{d} } :=\mathrm{span}\{P\psi _{\delta _{i},z _{i}}^{j}\ :\
i=1,\dots,\kappa,\ j=0,1,\dots ,n\},
\end{equation*}%
\begin{equation*}
K_{\mathbf{z},\mathbf{d} } ^{\perp }:=\left\{ \phi \in \mathfrak{H}_\eps\ :\ (\phi ,P\psi _{\delta _{i},z _{i}}^{j})=0,\ i=1,\dots,\kappa,\ j=0,1,\dots ,n\right\}
\end{equation*}%
and  the  projection operators
\begin{equation*}
\Pi _{\mathbf{z},\mathbf{d} }(u):=\sum\limits_{i=1}^{\kappa}\sum%
\limits_{j=0}^{n}(u,P\psi _{\delta _{i},z _{i}}^{j})P\psi _{\delta _{i},z _{i}}^{j}\quad \hbox{and}\quad \Pi _{%
\mathbf{z},\mathbf{d} }^{\perp }(u):=u-\Pi _{\mathbf{z},\mathbf{d},%
  }(u).
\end{equation*}

Our approach to solve problem \eqref{rep} will be to solve the system
\begin{equation}\label{lj-sc}
\left\{
\begin{aligned}
&\Pi _{\mathbf{z},\mathbf{d} }^{\perp }\Big\{V_{\mathbf{z},\mathbf{%
d} }+\phi -i^{\ast }\left[ f_{\epsilon }(V_{\mathbf{z},\mathbf{%
d} }+\phi )\right] \Big\} =0, \\
&\Pi _{\mathbf{z},\mathbf{d} }\Big\{ V_{\mathbf{z},\mathbf{d},%
  }+\phi -i^{\ast }\left[  f_{\epsilon }(V_{\mathbf{z},\mathbf{d},%
  }+\phi )\right] \Big\} =0,  \\
&(\mathbf{z},\mathbf{d} )\in \Lambda \quad\hbox{and}\quad \phi \in K_{\mathbf{z},\mathbf{d} }^{\perp }.\\ \end{aligned}
  \right.
\end{equation}

\subsection{The reduction argument}

The third  step   is reducing the problem to a finite dimensional one.
\medskip

 The first result we need concerns the invertibility of the linear operator    $\mathfrak{L}_{\mathbf{z},\mathbf{d} }:K_{\mathbf{z},\mathbf{d%
} }^{\perp }\rightarrow K_{\mathbf{z},\mathbf{d} %
}^{\perp }$ defined by
\begin{equation*}
\mathfrak{L}_{\mathbf{z},\mathbf{d} }\phi :=\phi -\Pi _{\mathbf{z},\mathbf{d},%
  }^{\perp }i^{\ast }\left[ f^{\prime }_0(V_{\mathbf{z},%
\mathbf{d} })\phi \right] .
\end{equation*}%

We will prove the following.

\begin{proposition}
\label{pro1} For any compact subset $\mathbf{C}$ of $\Lambda $ there exist $%
\epsilon _{0}>0$ and $c>0$ such that for each $\epsilon \in (0,\epsilon
_{0}) $ and $(\mathbf{z},\mathbf{d} )\in \mathbf{C}$ the operator $%
L_{\mathbf{z},\mathbf{d} }$ is invertible and
\begin{equation*}
\left\Vert L_{\mathbf{z},\mathbf{d} }\phi \right\Vert \geq
c\left\Vert \phi \right\Vert \ \quad \ \forall \ \phi \in K_{\mathbf{z},%
\mathbf{d} }^{\perp }.
\end{equation*}
\end{proposition}

Secondly, we solve the first equation in system \eqref{lj-sc}, namely
for each $(\mathbf{z},\mathbf{d} )\in \Lambda
$ and small $\epsilon $ we find  a function $\phi \in K_{\mathbf{z},\mathbf{d},%
  }^{\perp }$ which solves the first equation in system \eqref{lj-sc}.
To do this, we use a simple contraction mapping argument together with
 the estimate of the rest term given in Proposition \ref{err}.

\begin{proposition}
\label{pro2} For any compact subset $\mathbf{C}$ of $\Lambda $ there exist $%
\epsilon _{0}>0$ and $c>0$ such that for each $\epsilon \in (0,\epsilon
_{0}) $ and $(\mathbf{z},\mathbf{d} )\in \mathbf{C}$ there exists
a unique $\phi _{\mathbf{z},\mathbf{d} }^{\epsilon }\in K_{\mathbf{%
s},\mathbf{d} }^{\perp }$ which solves the first equation in system \eqref{lj-sc} and satisfies
$$
\left\Vert \phi _{\mathbf{z},\mathbf{d} }^{\epsilon }\right\Vert
\leq c    |\eps|^{\eta} , \label{eq2-pro2-1}
$$
where $\eta$ is given in Proposition \ref{err}.
\end{proposition}

Finally, we reduce  the problem to a finite dimensional one.
\medskip
We introduce the energy functional $J_{\epsilon }: \mathfrak{H}_\eps\rightarrow \mathbb{R}$ defined by
\begin{equation*}
J_{\epsilon }(u):={\frac{1}{2}}\int\limits_{\Omega } |\nabla u|^{2}dx- \int\limits_{\Omega } F_\eps(u),
\end{equation*}%
 where

\begin{itemize}
\item[(i)] {\em  for problem $(\mathcal{BN})_\eps$}
$$ F_\eps(u):={1\over p+1 }|u|^{p+1 } +{1\over2}\eps u^2 \quad \hbox{and}\ \mathfrak{H}_\eps:=H^1_0(\Omega)  $$

\item[(ii)] {\em   for problem $(\mathcal{AC})_\eps$}
$$F_\eps(u):={1\over p+1-\eps}|u|^{p+1 -\eps} \quad \hbox{and}\  \mathfrak{H}_\eps:=H^1_0(\Omega)\cap L^{s_\varepsilon}(\Omega)
$$

\item[(iii)] {\em   for problem $(\mathcal{C})_\eps$}
$$F_\eps(u):={1\over p+1 }|u|^{p+1  } \quad \hbox{and}\  \mathfrak{H}_\eps:=H^1_0(\Omega_\eps).
$$

\end{itemize}

It is well known that critical points of $J_{\epsilon }$ are the solutions to problem \eqref{rep}. We introduce
the {\em reduced energy} $\widetilde{J}_{\epsilon }:\Lambda \rightarrow
\mathbb{R}$ by
\begin{equation*}
\widetilde{J}_{\epsilon }(\mathbf{z},\mathbf{d} ):=J_{\epsilon
}(V_{\mathbf{z},\mathbf{d} }+\phi _{\mathbf{z},\mathbf{d}  }^{\epsilon })
\end{equation*}%
We prove that   critical points of $\widetilde{J}_{\epsilon }$ generates 
  solutions to the second equation in  system \eqref{lj-sc} and so blowing-up solutions to problem \eqref{rep}.

\begin{proposition}
\label{pro3} The function $V_{\mathbf{z},\mathbf{d} }+\phi _{%
\mathbf{z},\mathbf{d} }^{\epsilon }$ is a critical point of the
functional $J_{\epsilon }$ if and only if the point $(\mathbf{z},\mathbf{d}
  )$ is a critical point of the function $\widetilde{J}_{\epsilon }.$
\end{proposition}

\subsection{The reduced problem}
The last step is  looking for critical points of the   {\em reduced energy} $\widetilde{%
J}_{\epsilon }.$
\medskip

To to this,   we need an accurate asymptotic expansion of the   {\em reduced energy}  $\widetilde{%
J}_{\epsilon },$

\begin{proposition}
\label{pro4-1}It holds true that
\begin{align}\label{otto}
& \widetilde{J}_{\epsilon }(\mathbf{z},\mathbf{d} )=a(\eps)+b|\eps|^\gamma \Phi(\mathbf{z},\mathbf{d} )+o\(|\epsilon |^{\gamma}\)
 \end{align}%
$C^{1}$-uniformly on compact sets of $\Lambda .$
Here $a(\eps)$ is constant which depends only on $n,$ $\kappa$ and $\eps $ and $b$ is a constant.
The positive constant $\gamma$ and the function $\Phi$ are defined as follows.
\begin{itemize}
\item[] { In the case $\mathfrak{M}-(\mathcal{BN})_\eps:$}\qquad
$\gamma:= {n-2\over n-4}$\qquad  and
\begin{equation}\label{12} \Phi ( \mathbf{z},\mathbf{d} ):=
\left\{
\begin{aligned}
& \sum\limits_{i=1}^\kappa d_i^{n-2}H(z_i,z_i)+\sum\limits_{i=1}^\kappa \lambda_i\lambda_j\(d_id_j\)^{n-2\over2}G(z_i,z_j)-
  \sum\limits_{i=1}^\kappa   d_i^2 &\hbox{if $\eps>0$}\\
 &\sum\limits_{i=1}^\kappa d_i^{n-2}H(z_i,z_i)+\sum\limits_{i=1}^\kappa \lambda_i\lambda_j\(d_id_j\)^{n-2\over2}G(z_i,z_j)+
  \sum\limits_{i=1}^\kappa   d_i^2 &\hbox{if $\eps<0$}\\
\end{aligned}  \right.
\end{equation}
$$ $$
\item[] { In the case $\mathfrak{M}-(\mathcal{AC})_\eps:$}\qquad
  $\gamma:= {1}$ \qquad and
\begin{equation}\label{34} \Phi ( \mathbf{z},\mathbf{d} ):=
\left\{
\begin{aligned}
& \sum\limits_{i=1}^\kappa d_i^{n-2}H(z_i,z_i)+\sum\limits_{i=1}^\kappa \lambda_i\lambda_j\(d_id_j\)^{n-2\over2}G(z_i,z_j)-
  \sum\limits_{i=1}^\kappa  \ln d_i  &\hbox{if $\eps>0$}\\
 &\sum\limits_{i=1}^\kappa d_i^{n-2}H(z_i,z_i)+\sum\limits_{i=1}^\kappa \lambda_i\lambda_j\(d_id_j\)^{n-2\over2}G(z_i,z_j)+
  \sum\limits_{i=1}^\kappa \ln  d_i  &\hbox{if $\eps<0$}\\
\end{aligned}  \right.
\end{equation}

 $$ $$
\item[] { In the case $\mathfrak{T}-(\mathcal{AC})_\eps :$}\qquad
$\gamma:= {1}$\qquad and
\begin{equation}\label{5}\Psi ( \mathbf{z},\mathbf{d} ):= H(z ,z )d_1^{n-2}- \sum\limits_{i=1}^{\kappa-1}{1\over(1+\sigma_i^2)^{n-2\over2}}\({d_{i+1}\over d_i}\)^{n-2\over2}
- \sum\limits_{j=1}^\kappa\ln d_i.\end{equation}

$$ $$
\item[] { In the case $\mathfrak{T}-(\mathcal{C})_\eps:$}\qquad
  $\gamma:= {{n-2\over2\kappa}}$\qquad and
\begin{equation}\label{6}\Psi ( \mathbf{z},\mathbf{d} ):= H(z_0,z_0)d_1^{n-2}+{1\over(1+\sigma_k^2)^{n-2}}{1\over d_k^{n-2}}+ \sum\limits_{j=1}^{\kappa-1}{1\over(1+\sigma_j^2)^{n-2\over2}}\({d_{j+1}\over d_j}\)^{n-2\over2}.\end{equation}

\end{itemize}

\end{proposition}

\bigskip
Finally, we   reduce  the problem of finding blowing-up solutions to the problem \eqref{rep} to the problem of finding {\em good} critical points of the  function $\Psi,$ which is defined
on a finite dimensional space.

\begin{theorem}
\label{finale}
Assume $\(\mathbf{z}^*,\mathbf{d}^* \)\in\Lambda$ is a $C^1-$stable critical points of the function $\Psi$. Then if $\eps$ is small enough
there exists a solution $u_\eps$ to problem \eqref{rep} such that
$$\left\|u_\eps - V_ {\mathbf{z}_\eps,\mathbf{d}_\eps }\right\|\to 0\quad\hbox{as}\ \eps\to0$$
with
$$\(\mathbf{z}_\eps,\mathbf{d}_\eps \)\to \(\mathbf{z}^*,\mathbf{d}^*\)\quad\hbox{as}\ \eps\to0$$
\end{theorem}
\begin{proof}
It follows immediately from the expansion \eqref{otto} and Proposition \ref{pro3}.
\end{proof}

We remark that
$\(\mathbf{z}^*,\mathbf{d}^* \)\in\Lambda$ is a $C^1-$stable critical points of the function $\Psi$ if
\begin{itemize}
\item[(i)] $\(\mathbf{z}^*,\mathbf{d}^* \)\in\Lambda$ is an isolated minimum point of $\Psi,$
\item[(ii)] $\(\mathbf{z}^*,\mathbf{d}^* \)\in\Lambda$ is an isolated maximum point of $\Psi,$
\item[(iii)] $\(\mathbf{z}^*,\mathbf{d}^* \)\in\Lambda$ is a non degenerate critical  point of $\Psi.$
\item[(iv)] $\(\mathbf{z}^*,\mathbf{d}^* \)\in\Lambda$ is a critical point of   min-max  type      of $\Psi $ (according to the Definition given by Del Pino-Felmer-Musso in \cite{delfemu3})

\end{itemize}
 
\section{Examples}
\subsection{Multi-bubbles for the Brezis-Nirenberg problem $(\mathcal{BN})_\eps$ and  the ''slightly sub-critical'' problem $(\mathcal{AC})_\eps$ when $\eps$ is positive}

According to Theorem \ref{finale}, solutions to these problems are generated by  $C^1-$stable critical points of the function $\Psi$ defined in \eqref{12} and \eqref{34}
when $\eps>0.$

\begin{itemize}
\item[$\bullet$] $\kappa=1$

The function $\Psi$ has a     minimum point  $\Rightarrow$  if  $\eps\sim0^+$  problems $(\mathcal{BN})_\eps$ and $(\mathcal{AC})_\eps$  have
a positive solution with one blow-up point.
$$ $$
\item[$\bullet$] $\kappa=2$, $\lambda_1=\lambda_2=+1,$   $\Omega$ {\em is a dumb-bell with a thin handle}

The function $\Psi$ has a     minimum point  $\Rightarrow$   if  $\eps\sim0^+$  problems $(\mathcal{BN})_\eps$ and $(\mathcal{AC})_\eps$  have
a positive solution with two different blow-up points.
$$ $$
\item[$\bullet$] $\kappa\ge2$, $\lambda_1=\dots=\lambda_\kappa=+1,$   $\Omega$ {\em is convex}

The function $\Psi$ does not have any critical      points  $\Rightarrow$   if  $\eps\sim0^+$  problems $(\mathcal{BN})_\eps$ and $(\mathcal{AC})_\eps$ do not  have
a positive solution with $\kappa$ different blow-up points.
$$ $$
\item[$\bullet$] $\kappa=2$, $\lambda_1=+1,$ $ \lambda_2=-1$

The function $\Psi$ has a    minimum point and $(n-1)$ critical points of min-max type $\Rightarrow$ if  $\eps\sim0^+$  problems $(\mathcal{BN})_\eps$ and $(\mathcal{AC})_\eps$  have
$n$ sign changing solution with one positive and one negative blow-up points.
\end{itemize}

\subsection{Multi-bubbles for the Brezis-Nirenberg problem $(\mathcal{BN})_\eps$ and  the ''slightly super-critical'' problem $(\mathcal{AC})_\eps$ when $\eps$ is negative}
According to Theorem \ref{finale}, solutions to these problems are generated by  $C^1-$stable critical points of the function $\Psi$ defined in \eqref{12} and \eqref{34}
when $\eps<0.$
\begin{itemize}
\item[$\bullet$] $\kappa=1$

The function $\Psi$ does not have any critical      points  $\Rightarrow$  if  $\eps\sim0^-$  problems $(\mathcal{BN})_\eps$ and $(\mathcal{AC})_\eps$ do  not have
any positive solution with one blow point.
$$ $$
\item[$\bullet$] $\kappa=2$, $\lambda_1=\lambda_2=+1 $   or  $\kappa=3$, $\lambda_1=\lambda_2=\lambda_3=+1 ,$  $\Omega$ {\em has a hole}

The function $\Psi$ has a     critical  point of min-max type $\Rightarrow$   if  $\eps\sim0^-$  problems $(\mathcal{BN})_\eps$ and $(\mathcal{AC})_\eps$  have
a positive solution with two different positive blow-up points.
$$ $$
\item[$\bullet$] $\kappa=2$, $\lambda_1=+1,$ $ \lambda_2=-1$

The function $\Psi$  does not have any critical      points  $\Rightarrow$  if  $\eps\sim0^-$  problems $(\mathcal{BN})_\eps$ and $(\mathcal{AC})_\eps$ do  not have
any sign changing solution with one positive and one negative blow-up points.

\end{itemize}
 \subsection{Tower of bubbles with alternating sign for the ''almost critical'' problem $(\mathcal{AC})_\eps$ when $\eps$ is negative }
 According to Theorem \ref{finale}, solutions to this problem  are generated by  $C^1-$stable critical points of the function $\Psi$ defined in \eqref{5}.
 
\begin{itemize}
\item[$\bullet$] $\kappa\ge1$

The function $\Psi$ has a      critical  point of min-max type $\Rightarrow$   if  $\eps\sim0^+$  problems $(\mathcal{AC})_\eps$ has
a sign changing solution with $\kappa$ collapsing  blow-up points with alternating sign.
\end{itemize}

\subsection{Tower of bubbles with alternating sign for the Coron's problem $(\mathcal{ C})_\eps$  }
According to Theorem \ref{finale}, solutions to this problem  are generated by  $C^1-$stable critical points of the function $\Psi$ defined in \eqref{6}.

\begin{itemize}
\item[$\bullet$] $\kappa\ge1$

The function $\Psi$ has a   critical  point of min-max type   $\Rightarrow$   if  $\eps\sim0^+$  problems $(\mathcal{C})_\eps$ has
a sign changing solution with $\kappa$ collapsing  blow-up points with alternating sign.
\end{itemize}

 \end{document}